\documentclass[11pt, amsfonts, dvipdfm]{amsart}

\usepackage{amsmath,amssymb,amsthm}
\usepackage{eucal}
\usepackage[colorlinks=true,backref=page]{hyperref}
\usepackage[all]{xy}

\numberwithin{equation}{section}

\SelectTips{eu}{12}

\newtheorem{theorem}{Theorem}[section]
\newtheorem{proposition}[theorem]{Proposition}
\newtheorem{lemma}[theorem]{Lemma}
\newtheorem{corollary}[theorem]{Corollary}

\theoremstyle{definition}

\newtheorem{question}[theorem]{Question}

\theoremstyle{remark}

\numberwithin{equation}{section}

\newcommand{\Z}{\mathbb{Z}}
\newcommand{\Q}{\mathbb{Q}}

\newcommand{\C}{\mathbb{C}}
\newcommand{\Hom}{\mathrm{Hom}}
\newcommand{\map}{\mathrm{map}}
\newcommand{\GL}{\mathrm{GL}}
\newcommand{\hur}{\mathrm{hur}}

\SelectTips{cm}{}

\title[The space of commuting elements and classifying spaces]{The space of commuting elements in a Lie group and maps between classifying spaces}

\author[D. Kishimoto]{Daisuke Kishimoto}
\address{Faculty of Mathematics, Kyushu University, Fukuoka, 819-0395, Japan}
\email{kishimoto@math.kyushu-u.ac.jp}

\author[M. Takeda]{Masahiro Takeda}
\address{Institute for Liberal Arts and Sciences, Kyoto University, Kyoto, 606-8316, Japan}
\email{takeda.masahiro.87u@kyoto-u.ac.jp}

\author[M. Tsutaya]{Mitsunobu Tsutaya}
\address{Faculty of Mathematics, Kyushu University, Fukuoka, 819-0395, Japan}
\email{tsutaya@math.kyushu-u.ac.jp}

\date{\today}

\subjclass[2020]{55R37, 57T10}

\keywords{space of commuting elements, Lie group, classifying space}

\begin{document}

\maketitle

\begin{abstract}
  Let $\pi$ be a discrete group, and let $G$ be a compact connected Lie group. Then there is a map $\Theta\colon\Hom(\pi,G)_0\to\map_*(B\pi,BG)_0$ between the null-components of the spaces of homomorphisms and based maps, which sends a homomorphism to the induced map between classifying spaces. Atiyah and Bott studied this map for $\pi$ a surface group, and showed that it is surjective in rational cohomology. In this paper, we prove that the map $\Theta$ is surjective in rational cohomology for $\pi=\Z^m$ and the classical group $G$ except for $SO(2n)$, and that it is not surjective for $\pi=\Z^m$ with $m\ge 3$ and $G=SO(2n)$ with $n\ge 4$. As an application, we consider the surjectivity of the map $\Theta$ in rational cohomology for $\pi$ a finitely generated nilpotent group. We also consider the dimension of the cokernel of the map $\Theta$ in rational homotopy groups for $\pi=\Z^m$ and the classical groups $G$ except for $SO(2n)$.
\end{abstract}


\section{Introduction}\label{introduction}

Given two topological groups $G$ and $H$, there is a natural map
\[
  \widehat{\Theta}\colon\Hom(G,H)\to\map_*(BG,BH)
\]
sending a homomorphism to its induced map between classifying spaces, where $\Hom(G,H)$ and $\map_*(BG,BH)$ denote the spaces of homomorphisms and based maps, respectively. If $G$ and $H$ are discrete, then the map $\widehat{\Theta}$ in $\pi_0$ is a well-known bijection
\[
  \Hom(G,H)\cong[BG,BH]_*.
\]
However, the map $\widehat{\Theta}$ in $\pi_0$ is not bijective in general. Indeed, if $G=H=U(n)$, then Sullivan \cite{S} constructed a map between classifying spaces, called the unstable Adams operation, which is not in the image of the map $\widehat{\Theta}$ in $\pi_0$, even rationally. Since then, the map $\widehat{\Theta}$ in $\pi_0$ has been intensely studied for both $G$ and $H$ being Lie groups completed at a prime, which led to a new development of algebraic topology and has been producing a variety of applications. See surveys \cite{BLO,G,JMO} for details. Clearly, the map $\widehat{\Theta}$ is of particular importance not only in $\pi_0$. However, not much is known about higher homotopical structures of the map $\widehat{\Theta}$ such as homotopy groups and (co)homology of dimension $\ge 1$.

We describe two interpretation of the map $\widehat{\Theta}$. The first one is from algebraic topology. Stasheff \cite{St} introduced an $A_\infty$-map between topological monoids, which is defined by replacing the equality in the definition of a homomorphism by coherent higher homotopies with respect to the associativity of the multiplications. He also showed that to each $A_\infty$-map, we can assign a map between classifying spaces, and so there is a map
\[
  \mathcal{A}_\infty(G,H)\to\map_*(BG,BH)
\]
where $\mathcal{A}_\infty(G,H)$ denotes the space of $A_\infty$-maps between topological groups $G,H$. It is proved in \cite{F,T} that this map is a weak homotopy equivalence, and since the map $\widehat{\Theta}$ factors through this map, we can interpret that the map $\widehat{\Theta}$ depicts the difference of homomorphisms, solid objects, and $A_\infty$-maps, soft objects, between topological groups.

The second interpretation is from bundle theory. Let $\pi$ be a finitely generated discrete group, and let $G$ be a compact connected Lie group. Let $\Hom(\pi,G)_0$ and $\map_*(B\pi,BG)_0$ denote the path-components of $\Hom(\pi,G)$ and $\map_*(B\pi,BG)$ containing trivial maps, respectively. In this paper, we study the natural map
\[
  \Theta\colon\Hom(\pi,G)_0\to\map_*(B\pi,BG)_0
\]
which is the restriction of the map $\widehat{\Theta}$. If $B\pi$ has the homotopy type of a manifold $M$, then $\Hom(\pi,G)_0$ and $\map_*(B\pi,BG)_0$ are identified with the based moduli spaces of flat connections and all connections on the trivial $G$-bundle over $M$, denoted by $\mathrm{Flat}(M,G)_0$ and $\mathcal{C}(M,G)_0$, respectively. Under this identification, the map $\Theta$ can be interpreted as the inclusion
\[
  \mathrm{Flat}(M,G)_0\to\mathcal{C}(M,G)_0.
\]

Atiyah and Bott \cite{AB} studied the map $\Theta$ for a surface group $\pi$ through the above flat bundle interpretation in the context of gauge theory because flat connections are solutions to the Yang-Mills equation over a Riemann surface. In particular, they used Morse theory to prove that the map $\Theta$ is surjective in rational cohomology whenever $\pi$ is a surface group. However, their proof is so specialized to surface groups that it does not apply to other groups $\pi$. Then we ask:

\begin{question}
  Is the map $\Theta$ surjective in rational cohomology whenever $B\pi$ is of the homotopy type of a manifold?
\end{question}

In this paper, we study the above question in the special case $\pi=\Z^m$. The space $\Hom(\Z^m,G)$ is called the space of commuting elements in $G$ because there is a natural homeomorphism
\[
  \Hom(\Z^m,G)\cong\{(g_1,\ldots,g_m)\in G^m\mid g_ig_j=g_jg_i\text{ for }1\le i,j\le m\}
\]
where we will not distinguish these two spaces. Recently, several results on the space of commuting elements in a Lie group have been obtained from a view of algebraic topology \cite{AC,ACT-G,AGG,Ba,BJS,BR,C,Go,GPS,KT1,KT2,RS1,RS2}. In particular, the first and the second authors gave a minimal generating set of the rational cohomology of $\Hom(\Z^m,G)_0$ when $G$ is the classical group except for $SO(2n)$. Using this generating set, we will prove:

\begin{theorem}
  \label{main 1}
  If $G$ is the classical group except for $SO(2n)$, then the map
  \[
    \Theta\colon\Hom(\Z^m,G)_0\to\map_*(B\Z^m,BG)_0
  \]
  is surjective in rational cohomology.
\end{theorem}

As an application of Theorem \ref{main 1}, we will prove the following theorem. We refer to \cite{HMR} for the localization of nilpotent groups.

\begin{theorem}
  \label{nilpotent group}
  Let $\pi$ be a finitely generated nilpotent group, and let $G$ be the classical group except for $U(1)$ and $SO(2n)$. Then the map
  \[
    \Theta\colon\Hom(\pi,G)_0\to\map_*(B\pi,BG)_0
  \]
  is surjective in rational cohomology if and only if the rationalization $\pi_{(0)}$ is abelian.
\end{theorem}

As a corollary, we will obtain:

\begin{corollary}
  \label{nilmanifold}
  Let $M$ be a nilmanifold, and let $G$ be the classical group except for $SO(2n)$. Then the inclusion
  \[
    \mathrm{Flat}(M,G)_0\to\mathcal{C}(M,G)_0
  \]
  is surjective in rational cohomology if and only if $M$ is a torus.
\end{corollary}

As a corollary to Theorem \ref{main 1}, we will show that the map $\Theta$ is surjective in rational cohomology for $G=SO(2n)$ with $n=2,3$ (Corollary \ref{SO low rank}). On the other hand, as mentioned above, the result of Atiyah and Bott \cite{AB} implies that the map $\Theta$ is surjective in rational cohomology for $m=2$ and $G=SO(2n)$ with any $n\ge 2$. Then we may expect that the map $\Theta$ is also surjective in rational cohomology for $m\ge 3$ and $n\ge 4$. However, the surjectivity breaks as:

\begin{theorem}
  \label{main 2}
  For $m\ge 3$ and $n\ge 4$, the map
  \[
    \Theta\colon\Hom(\Z^m,SO(2n))_0\to\map_*(B\Z^m,BSO(2n))_0
  \]
  is not surjective in rational cohomology.
\end{theorem}

We will also consider the map $\Theta$ in rational homotopy groups. It is proved in \cite{KT1} that $\Hom(\Z^m,G)_0$ is rationally hyperbolic, and so the total dimension of its rational homotopy groups is infinite. On the other hand, we will see in Section \ref{rational homotopy groups} that the rational homotopy group of $\map_*(B\Z^m,BG)_0$ is finite dimensional. Then the map $\Theta$ for $\pi=\Z^m$ cannot be injective. On the other hand, we can consider the surjectivity of the map $\Theta$ in rational homotopy groups by looking at its cokernel. Baird and Ramras \cite{BR} gave a lower bound for the dimension of the cokernel of the map $\Theta$ for $G=\GL_n(\C)$ in rational homotopy groups. In particular, for $\pi=\Z^m$, they proved that the dimension of the cokernel of the map
\[
  \Theta_*\colon\pi_i(\Hom(\Z^m,\GL_n(\C))_0)\otimes\Q\to\pi_i(\map_*(B\Z^m,B\GL_n(\C))_0)\otimes\Q
\]
is bounded below by $\sum_{i<k\le n}\binom{m}{2i-k}$ whenever $n\ge\frac{m+i}{2}$. By using Theorem \ref{main 1}, we can improve this result as follows.

\begin{theorem}
  \label{main 3}
  Let $c_i(m,G)$ be the dimension of the cokernel of the map
  \[
    \Theta_*\colon\pi_i(\Hom(\Z^m,G)_0)\otimes\Q\to\pi_i(\map_*(B\Z^m,BG)_0)\otimes\Q.
  \]

  \begin{enumerate}
    \item For $G=U(n),SU(n)$, we have
    \[
      c_i(m,G)=\sum_{i<k\le n}\binom{m}{2i-k}.
    \]

    \item For $G=Sp(n),SO(2n+1)$, we have
    \[
      c_i(m,G)\ge\sum_{\frac{i}{3}<k\le n}\binom{m}{4i-k}
    \]
    where the equality holds for $i\le 2n+3$.
  \end{enumerate}
\end{theorem}

Remarks on Theorem \ref{main 3} are in order. By \cite{GPS}, $\pi_1(\Hom(\Z^m,G)_0)$ is abelian, and it is easy to see that $\pi_1(\map_*(B\Z^m,BG))$ is abelian too. Then $\pi_1\otimes\Q$ in Theorem \ref{main 3} makes sense. By \cite{Be}, the $G=U(n)$ case is equivalent to the $G=\GL_n(\C)$ case, and so we can see that the lower bound of Baird and Ramras \cite{BR} for $\pi=\Z^m$ mentioned above is attained by Theorem \ref{main 3}.

\subsection{Acknowledgements}

The authors were support in part by JSPS KAKENHI JP17K05248 and JP19K03473 (Kishimoto), JP21J10117 (Takeda), and JP19K14535 (Tsutaya).


\section{The map $\Phi$}\label{the map}

Hereafter, let $G$ be a compact connected Lie group with maximal torus $T$, and let $W$ denote the Weyl group of $G$. We define a map
\[
  \Phi\colon G/T\times_W T^m\to\Hom(\Z^m,G)_0
\]
by $\Phi(gT,(g_1,\ldots,g_m))=(gg_1g^{-1},\ldots,gg_mg^{-1})$ for $g\in G$ and $g_1,\ldots,g_m\in T$, where $T^m$ denotes the direct product of $m$ copies of $T$, instead of a torus of dimension $m$. In this section, we will define maps involving the map $\Phi$ and show their properties.

First, we recall the following result of Baird \cite{Ba}. It is well known that there is a natural isomorphism
\[
  H^*(G/T\times_W T^m;\Q)\cong H^*(G/T\times T^m;\Q)^W
\]
and so we will not distinguish them. Baird \cite{Ba} proved:

\begin{theorem}
  \label{Baird}
  The map
  \[
    \Phi^*\colon H^*(\Hom(\Z^m,G)_0;\Q)\to H^*(G/T\times T^m;\Q)^W
  \]
  is an isomorphism.
\end{theorem}

By using Theorem \ref{Baird}, the first and the second authors \cite{KT1} gave a minimal generating set of the rational cohomology of $\Hom(\Z^m,G)_0$ when $G$ is the classical group except for $SO(2n)$, which we recall in the next section.

In order to define maps involving the map $\Phi$, we need the functoriality of classifying spaces. Then we employ the Milnor construction \cite{M} as a model for the classifying space. Let
\[
  EG=\lim_{n\to\infty}\underbrace{G*\cdots*G}_n
\]
where $X*Y$ denotes the join of spaces $X$ and $Y$. Following Milnor \cite{M}, we denote a point of $EG$ by
\[
  t_1g_1\oplus t_2g_2\oplus\cdots
\]
such that $t_i\ge 0$, $\sum_{n\ge 1}t_n=1$ with only finitely many $t_i$ being non-zero, and $s_1g_1\oplus s_2g_2\oplus\cdots=t_1h_1\oplus t_2h_2\oplus\cdots$ if $s_k=t_k=0$ ($g_k\ne h_k$, possibly) and for $i\ne k$, $s_i=t_i$ and $g_i=h_i$, where $e$ denotes the identity element of $G$. Then $G$ acts freely on $EG$ by
\[
  (t_1g_1\oplus t_2g_2\oplus\cdots)\cdot g=t_1g_1g\oplus t_2g_2g\oplus\cdots.
\]
We define the classifying space of $G$ by
\[
  BG=EG/G.
\]

Note that the inclusion $ET\to EG$ induces a map $\iota\colon BT\to EG/T$ which is a homotopy equivalence because both $ET$ and $EG$ are contractible. We record a simple fact which follows immediately from the definition of the Milnor construction.

\begin{lemma}
  \label{BT->BG}
  The natural map $BT\to BG$ factors as the composite
  \[
    BT\xrightarrow{\iota}EG/T\to EG/G=BG.
  \]
\end{lemma}

Now we define a map
\[
  \phi\colon G/T\times BT\to BG,\quad(gT,[t_1g_1\oplus t_2g_2\oplus\cdots])\mapsto[t_1gg_1\oplus t_2gg_2\oplus\cdots]
\]
for $g\in G$ and $[t_1g_1\oplus t_2g_2\oplus\cdots]\in BT$. Since $T$ is abelian, we have
\[
  [t_1ghg_1\oplus t_2ghg_2\oplus\cdots]=[t_1gg_1h\oplus t_2gg_2h\oplus\cdots]=[t_1gg_1\oplus t_2gg_2\oplus\cdots]
\]
 for $h\in T$, implying that the map $\phi$ is well-defined. We also define
\[
  \bar{\alpha}\colon G/T\to EG/T,\quad gT\mapsto[1g\oplus 0e\oplus 0e\oplus\cdots].
\]
Let $\alpha$ denote the composite $G/T\xrightarrow{\bar{\alpha}}EG/T\xrightarrow{\iota^{-1}}BT$. Then there is a homotopy fibration $G/T\xrightarrow{\alpha}BT\to BG$.

\begin{lemma}
  \label{phi}
  There is a map $\widehat{\phi}\colon G/T\times BT\to BT$ satisfying the homotopy commutative diagram
  \[
    \xymatrix{
      G/T\vee BT\ar[r]\ar[d]^{\alpha\vee 1}&G/T\times BT\ar@{=}[r]\ar[d]^{\widehat{\phi}}&G/T\times BT\ar[d]^\phi\\
      BT\ar@{=}[r]&BT\ar[r]&BG.
    }
  \]
\end{lemma}

\begin{proof}
  Define a map
  \[
    \bar{\phi}\colon G/T\times BT\to EG/T,\quad(gT,[t_1g_1\oplus t_2g_2\oplus\cdots])\mapsto[t_1gg_1\oplus t_2gg_2\oplus\cdots].
  \]
  Quite similarly to the map $\phi$, we can see that the map $\bar{\phi}$ is well-defined. Let $\widehat{\phi}$ denote the composite
  \[
    G/T\times BT\xrightarrow{\bar{\phi}}EG/T\xrightarrow{\iota^{-1}}BT.
  \]
  Then by Lemma \ref{BT->BG}, the right square of the diagram in the statement is homotopy commutative. We also have
  \[
    \bar{\phi}(gT,[1e\oplus 0e\oplus 0e\oplus\cdots])=[1g\oplus 0e\oplus 0e\oplus\cdots]=\bar{\alpha}(gT)
  \]
  and
  \[
    \bar{\phi}(eT,[t_1g_1\oplus t_2g_2\oplus\cdots])=[t_1g_1\oplus t_2g_2\oplus\cdots]=\iota([t_1g_1\oplus t_2g_2\oplus\cdots]).
  \]
  for $[t_1g_1\oplus t_2g_2\oplus\cdots]\in BT$, where $[1e\oplus 0e\oplus 0e\oplus\cdots]$ is the basepoint of $BT$. Then the left square is homotopy commutative too, finishing the proof.
\end{proof}

We may think of the map $\widehat{\phi}$ as a higher version of the map defined by conjugation in \cite{Bo}. We define a map
\[
  \widehat{\Phi}\colon G/T\times_W\map_*(B\Z^m,BT)_0\to\map_*(B\Z^m,BG)_0
\]
by $\widehat{\Phi}(gT,f)(x)=\phi(gT,f(x))$ for $g\in G,\,f\in\map_*(B\Z^m,BT)_0$ and $x\in B\Z^m$. Since there is a natural homeomorphism $\Hom(\Z^m,T)_0\cong T^m$, we will not distinguish them.

\begin{lemma}
  \label{Phi}
  There is a commutative diagram
  \[
    \xymatrix{
      G/T\times_W\Hom(\Z^m,T)_0\ar[r]^(.57){\Phi}\ar[d]_{1\times\Theta}&\Hom(\Z^m,G)_0\ar[d]^\Theta\\
      G/T\times_W\map_*(B\Z^m,BT)_0\ar[r]^(.57){\widehat{\Phi}}&\map_*(B\Z^m,BG)_0.
    }
  \]
\end{lemma}

\begin{proof}
  By definition, we have
  \[
    \Theta(f)([t_1g_1\oplus t_2g_2\oplus\cdots])=[t_1f(g_1)\oplus t_2f(g_2)\oplus\cdots]
  \]
  for $f\in\Hom(\Z^m,T)_0$ and $[t_1g_1\oplus t_2g_2\oplus\cdots]\in B\Z^m$. Then we get
  \begin{align*}
    &\widehat{\Phi}\circ(1\times\Theta)(gT,f)([t_1g_1\oplus t_2g_2\oplus\cdots])\\
    &=[t_1gf(g_1)g^{-1}\oplus t_2gf(g_2)g^{-1}\oplus\cdots]\\
    &=\Theta\circ\Phi(gT,f)([t_1g_1\oplus t_2g_2\oplus\cdots])
  \end{align*}
  for $g\in G,\,f\in\Hom(\Z^m,T)_0$ and $[t_1g_1\oplus t_2g_2\oplus\cdots]\in B\Z^m$. Thus the proof is finished.
\end{proof}

Next, we consider the evaluation map
\[
  \omega\colon\map_*(X,Y)_0\times X\to Y,\quad(f,x)\mapsto f(x).
\]
Note that the map $\phi\colon G/T\times BT\to BG$ factors through $G/T\times_WBT$. We denote the map $G/T\times_WBT\to BG$ by the same symbol $\phi$.

\begin{lemma}
  \label{omega}
  There is a commutative diagram
  \[
    \xymatrix{
      G/T\times_W\map_*(B\Z^m,BT)_0\times B\Z^m\ar[r]^(.56){\widehat{\Phi}\times 1}\ar[d]_{1\times\omega}&\map_*(B\Z^m,BG)_0\times B\Z^m\ar[d]^\omega\\
      G/T\times_WBT\ar[r]^(.55)\phi&BG.
    }
  \]
\end{lemma}

\begin{proof}
  For $g\in G$, $f\in\map_*(B\Z^m,BT)_0$ and $x\in B\Z^m$, we have
  \[
    \omega\circ(\widehat{\Phi}\times 1)(gT,f,x)=\phi(gT,f(x))=\phi(gT,\omega(f,x))=\phi\circ(1\times\omega)(gT,f,x).
  \]
  Thus the statement is proved.
\end{proof}


\section{Rational cohomology}\label{rational cohomology}

In this section, we will prove Theorems \ref{main 1} and \ref{main 2}, and we will apply Theorem \ref{main 1} to prove Theorem \ref{nilpotent group}. To prove Theorem \ref{main 1}, we will employ the generating set of the rational cohomology of $\Hom(\Z^m,G)_0$ given in \cite{KT1}, and to prove Theorem \ref{main 2}, we will consider a specific element of $H^*(\Hom(\Z^m,SO(2n))_0)\cong H^*(SO(2n)/T\times T^m)^W$.


\subsection{Cohomology generators}

Hereafter, the coefficients of (co)homology will be always in $\Q$, and we will suppose that $G$ is of rank $n$, unless otherwise is specified. First, we set notation on cohomology. Since $G$ is of rank $n$, the cohomology of $BT$ is given by
\[
  H^*(BT)=\Q[x_1,\ldots,x_n],\quad|x_i|=2.
\]
We also have that the cohomology of $T^m$ is given by
\[
  H^*(T^m)=\Lambda(y_1^1,\ldots,y^1_n,\ldots,y^m_1,\ldots,y_n^m),\quad|y_i^j|=1
\]
such that $y_i^k=\pi_k^*(\overset{k}{\sigma(x_i)})$, where $\pi_k\colon B\Z^m\to B\Z$ is the $k$-th projection and $\sigma$ denotes the cohomology suspension. Let $[m]=\{1,2,\ldots,m\}$. For $I=\{i_1<\cdots<i_k\}\subset[m]$, we set
\[
  y_i^I=y_i^{i_1}\cdots y_i^{i_k}.
\]
It is well known that the map $\alpha\colon G/T\to BT$ induces an isomorphism
\[
  H^*(G/T)\cong H^*(BT)/(\widetilde{H}^*(BT)^W).
\]
We denote $\alpha^*(x_i)$ by the same symbol $x_i$, and so $H^*(G/T)$ is generated by $x_1,\ldots,x_n$.

Now we recall the minimal generating set of the rational cohomology of $\Hom(\Z^m,G)_0$ given in \cite{KT1}. For $d\ge 1$ and $I\subset[m]$, we define
\[
  z(d,I)=x_1^{d-1}y_1^I+\cdots+x_n^{d-1}y_n^I\in H^*(G/T\times T^m)
\]
and let
\[
  \mathcal{S}(m,U(n))=\{z(d,I)\mid d\ge 1,\,\emptyset\ne I\subset[m],\,d+|I|-1\le n\}
\]
where we have $|z(d,I)|=2d+|I|-2$. We also let
\[
  \mathcal{S}(m,SU(n))=\{z(d,I)\in\mathcal{S}(m,U(n))\mid d\ge 2\text{ or }|I|\ge 2\}
\]
where $x_1+\cdots+x_n=0$ and $y_1^i+\cdots+y_n^i=0$ for $i=1,\ldots,m$. Since $W$ is the symmetric group on $[n]$ for $G=U(n),SU(n)$ such that for $\sigma\in W$, $\sigma(x_i)=x_{\sigma(i)}$ and $\sigma(y_i^j)=y_{\sigma(i)}^j$, we have
\[
  \mathcal{S}(m,G)\subset H^*(G/T\times T^m)^W.
\]
For an integer $k$, let $\epsilon(k)=0$ for $k$ even and $\epsilon(k)=1$ for $k$ odd. We define
\[
  w(d,I)=x_1^{2d+\epsilon(|I|)-2}y_1^I+\cdots+x_n^{2d+\epsilon(|I|)-2}y_n^I\in H^*(G/T\times T^m)
\]
and let
\[
  \mathcal{S}(m,Sp(n))=\{w(d,I)\mid d\ge 1,\,\emptyset\ne I\subset[m],\,2d+|I|+\epsilon(|I|)-2\le 2n\}
\]
where we have $|w(d,I)|=4d+|I|+2\epsilon(|I|)-4$. We also let
\[
  \mathcal{S}(m,SO(2n+1))=\mathcal{S}(m,Sp(n)).
\]
Since $W$ is the signed symmetric group on $[n]$ for $G=Sp(n),SO(2n+1)$ such that for $\sigma\in W$, $(\pm\sigma)(x_i)=\pm x_{\sigma(i)}$ and $(\pm\sigma)(y_i^j)=\pm y_{\sigma(i)}^j$, we have
\[
  \mathcal{S}(m,G)\subset H^*(G/T\times T^m)^W.
\]
The following theorem is proved in \cite{KT1}.

\begin{theorem}
  \label{generator}
  If $G$ is the classical group except for $SO(2n)$, $(\Phi^*)^{-1}(\mathcal{S}(m,G))$ is a minimal generating set of the rational cohomology of $\Hom(\Z^m,G)_0$.
\end{theorem}


\subsection{Proof of Theorem \ref{main 1}}

First, we consider the map $\widehat{\phi}\colon G/T\times BT\to BT$ of Lemma \ref{phi} in cohomology.

\begin{lemma}
  \label{hat-Phi}
  For each $x_i\in H^*(BT)$, we have
  \[
    \widehat{\phi}^*(x_i)=x_i\times 1+1\times x_i.
  \]
\end{lemma}

\begin{proof}
  The statement immediately follows from the left square of the homotopy commutative diagram in Lemma \ref{phi}.
\end{proof}

Next, we consider the map $\Theta\colon\Hom(\Z^m,T)_0\to\map_*(B\Z^m,BT)_0$.

\begin{lemma}
  \label{H^1 rank 1}
  The map $\Theta\colon\Hom(\Z,T)_0\to\map_*(B\Z,BT)_0$ is a homotopy equivalence.
\end{lemma}

\begin{proof}
  For a topological group $K$ with a non-degenerate unit, there is a homeomorphism $(K*K)/K\cong\widetilde{\Sigma}K$ such that the composite
  \[
    \widetilde{\Sigma}K\simeq(K*K)/K\to BK
  \]
  is identified with the adjoint of the natural homotopy equivalence $K\simeq\Omega BK$, where $\widetilde{\Sigma}$ denotes the unreduced suspension. By definition, $\widetilde{\Sigma}\Z$ is homotopy equivalent to a wedge of infinitely many copies of $S^1$, and the map $\widetilde{\Sigma}\Z\to B\Z$ is identified with the fold map onto $S^1$. Thus the composite $\widetilde{\Sigma}\{0,1\}\to\widetilde{\Sigma}\Z\to B\Z$ is a homotopy equivalence. Note that for any homomorphism $f\colon\Z\to T$, there is a commutative diagram
  \[
    \xymatrix{
      \widetilde{\Sigma}\{0,1\}\ar[r]\ar[d]&\widetilde{\Sigma}\Z\ar[r]^{\widetilde{\Sigma}f}\ar[d]&\widetilde{\Sigma}T\ar[d]\\
      B\Z\ar@{=}[r]&B\Z\ar[r]^{Bf}&BT.
    }
  \]
  Thus since $\Hom(\Z,T)_0=\map_*(\{0,1\},T)_0$, we get a commutative diagram
  \[
    \xymatrix{
      \map_*(\{0,1\},T)_0\ar[r]^(.46){\widetilde{\Sigma}}\ar@{=}[d]&\map_*(\widetilde{\Sigma}\{0,1\},\widetilde{\Sigma}T)_*\ar[r]&\map_*(\widetilde{\Sigma}\{0,1\},BT)_0\\
      \Hom(\Z,T)_0\ar[rr]^\Theta&&\map_*(B\Z,BT)_0\ar[u]_\simeq.
    }
  \]
  Clearly, the composite of the top maps is identified with the homotopy equivalence $\map_*(\{0,1\},\Omega BT)_0\cong\map_*(\Sigma\{0,1\},BT)_0$. Then the bottom map is a homotopy equivalence too, completing the proof.
\end{proof}

\begin{lemma}
  \label{H^1}
  The map $\Theta\colon\Hom(\Z^m,T)_0\to\map_*(B\Z^m,BT)_0$ is a homotopy equivalence.
\end{lemma}

\begin{proof}
  Let $F_m$ be the free group of rank $m$. Clearly, we have
  \[
    \Hom(F_m,T)_0\cong(\Hom(\Z,T)_0)^m.
  \]
  Since $BF_m$ is homotopy equivalent to a wedge of $m$ copies of $S^1$, we also have
  \[
    \map_*(BF_m,BT)_0\simeq(\map_*(B\Z,BT)_0)^m.
  \]
  It is easy to see that through these equivalences, the map $\Theta\colon\Hom(F_m,T)_0\to\map_*(BF_m,BT)_0$ is identified with the product of $m$ copies of the map $\Theta\colon\Hom(\Z,T)_0\to\map_*(B\Z,BT)_0$. Thus by Lemma \ref{H^1 rank 1}, the map $\Theta\colon\Hom(F_m,T)_0\to\map_*(BF_m,BT)_0$ is a homotopy equivalence. Now we consider the commutative diagram
  \[
    \xymatrix{
      \Hom(\Z^m,T)_0\ar[r]^(.45)\Theta\ar[d]&\map_*(B\Z^m,BT)_0\ar[d]\\
      \Hom(F_m,T)_0\ar[r]^(.45)\Theta&\map_*(BF_m,BT)_0
    }
  \]
  induced from the abelianization $F_m\to\Z^m$. Since $T$ is abelian, the left map is a homeomorphism. Since the cofiber of the map $BF_m\to B\Z^m$ is simply-connected, the right map is a homotopy equivalence. Thus the top map is a homotopy equivalence too, completing the proof.
\end{proof}

We consider the evaluation map $\omega\colon\map_*(B\Z^m,BT)_0\times B\Z^m\to BT$ in cohomology. Since $B\Z^m$ is homotopy equivalent to the $m$-dimensional torus, we have
\[
  H^*(B\Z^m)=\Lambda(t_1,\ldots,t_m),\quad|t_i|=1.
\]
For $I=\{i_1<\cdots<i_k\}\subset[m]$, let
\[
  t_I=t_{i_1}\cdots t_{i_k}.
\]

\begin{lemma}
  \label{omega-Phi}
  For each $x_i\in H^*(BT)$, we have
  \[
    (\omega\circ(\Theta\times 1))^*(x_i)=y_i^1\times t_1+\cdots+y_i^m\times t_m.
  \]
\end{lemma}

\begin{proof}
  For the evaluation map $\omega\colon\map_*(B\Z,BT)_0\times B\Z\to BT$, we have
  \[
    \omega^*(x_1)=y_1^1\times t_1
  \]
  as in \cite{KK}, where we identify $\map_*(B\Z,BT)_0$ with $T$. By Lemma \ref{H^1}, we may assume $\Theta^*(y_1^1)=y_1^1$. Let $\iota_i\colon B\Z\to B\Z^m$ and $\pi_i\colon B\Z^m\to B\Z$ denote the $i$-th inclusion and the $i$-th projection, respectively. Since $\omega\circ(\pi_i^*\times\iota_j)$ is trivial for $i\ne j$, $\omega^*(x_k)$ is a linear combination of $\pi_k^*(y_1^1)\times t_1,\ldots,\pi_k^*(y_1^m)\times t_m$. There is a commutative diagram
  \[
    \xymatrix{
      \map_*(B\Z,BT)_0\times B\Z\ar[r]^(.72)\omega\ar[d]_{\pi_i^*\times\iota_i}&BT\ar@{=}[d]\\
      \map_*(B\Z^m,BT)_0\times B\Z^m\ar[r]^(.75)\omega&BT.
    }
  \]
  Then we get
  \begin{align*}
    (\Theta\times 1)^*\circ\omega^*(x_k)&=(\Theta\times 1)^*(\pi_k^*(y_1^1)\times t_1+\cdots+\pi_k^*(y_1^m)\times t_m)\\
    &=y_k^1\times t_1+\cdots+y_k^m\times t_m.
  \end{align*}
  Thus the proof is finished.
\end{proof}

Next, we consider the evaluation map $\omega\colon\map_*(B\Z^m,BG)_0\times B\Z^m\to BG$ in cohomology. Recall that the rational cohomology of $BG$ is given by
\[
  H^*(BG)=\Q[z_1,\ldots,z_n].
\]
We choose generators $z_1,\ldots,z_n$ as
\[
  j^*(z_i)=
  \begin{cases}
    x_1^i+\cdots+x_n^i&G=U(n)\\
    x_1^{2i}+\cdots+x_n^{2i}&G=Sp(n),SO(2n+1)
  \end{cases}
\]
and set $H^*(BSU(n))=H^*(BU(n))/(z_1)$, where $j\colon BT\to BG$ denotes the natural map. For $i=1,\ldots,n$ and $\emptyset\ne I\subset[m]$, we define $z_{i,I}\in H^*(\map_*(B\Z^m,BG)_0)$ by
\[
  \omega^*(z_i)=\sum_{\emptyset\ne I\subset[m]}z_{i,I}\times t_I
\]
where $z_{i,I}=1$ for $|z_i|=|I|$ and $z_{i,I}=0$ for $|z_i|<|I|$.

\begin{proposition}
  \label{H(map)}
  The rational cohomology of $\map_*(B\Z^m,BG)_0$ is a free commutative graded algebra generated by
  \[
    \mathcal{S}=\{z_{i,I}\mid 1\le i\le n,\,\emptyset\ne I\subset[m],\,|z_i|>|I|\}.
  \]
\end{proposition}

\begin{proof}
  Since the rationalization of $BG$ is homotopy equivalent to a product of Eilenberg-MacLane spaces, so is the rationalization of $\map_*(B\Z^m,BG)_0$. Then the cohomology of $\map_*(B\Z^m,BG)_0$ is a free commutative algebra. The rest can be proved quite similarly to \cite[Proposition 2.20]{AB}.
\end{proof}

We compute $\Theta^*(z_{i,I})$ for the classical group $G$ except for $SO(2n)$.

\begin{proposition}
  \label{Theta(z)}
  For $i=1,\ldots,n$ and $\emptyset\ne I\subset[m]$, if $|z_i|>|I|$, then
  \[
    \Phi^*\circ\Theta^*(z_{i,I})=
    \begin{cases}
      \displaystyle\frac{i!}{(i-|I|)!}z(i-|I|+1,I)&G=U(n),SU(n)\vspace{1mm}\\
      \displaystyle\frac{(2i)!}{(2i-|I|)!}w(i-\tfrac{|I|+\epsilon(|I|)}{2}+1,I)&G=Sp(n),SO(2n+1).
    \end{cases}
  \]
\end{proposition}

\begin{proof}
  First, we prove the $G=U(n)$ case. By Lemmas \ref{phi} and \ref{hat-Phi}, we have
  \[
    \phi^*(z_i)=\widehat{\phi}^*(j^*(z_i))=\widehat{\phi}^*(x_1^i+\cdots+x_n^i)=\sum_{k=1}^n(x_k\times 1+1\times x_k)^i.
  \]
  By Lemmas \ref{Phi} and \ref{omega}, there is a homotopy commutative diagram
  \[
    \xymatrix{
      G/T \times T^m \times B\Z^m \ar^-{1 \times \Theta \times 1}[r] \ar^{\Phi \times 1}[d] & G/T  \times \map_*(B\Z^m,BT)_0 \times B\Z^m \ar^-{1\times \omega}[r]\ar^{\widehat{\Phi} \times 1}[d]& G/T\times BT \ar^{\phi}[d]\\
      \Hom(\Z^m,G) \times B\Z^m \ar^-{\Theta\times 1}[r]& \map_*(B\Z^m,BT)_0 \times B\Z^m \ar^-{\omega}[r] & BG.
    }
  \]
  Then by Lemma \ref{omega-Phi}, we get
  \begin{align*}
    &(\Phi\times 1)^*\circ(\Theta\times 1)^*\circ\omega^*(z_i)\\
    &=(1\times\Theta\times 1)^*\circ(\widehat{\Phi}\times 1)^*\circ\omega^*(z_i)\\
    &=(1\times\Theta\times 1)^*\circ(1\times\omega)^*\circ\phi^*(z_i)\\
    &=(1\times\Theta\times 1)^*\circ(1\times\omega)^*\left(\sum_{k=1}^n(x_k\times 1+1\times x_k)^i\right)\\
    &=\sum_{k=1}^n(x_k\times 1+y_k^1\times t_1+\cdots+y_k^m\times t_m)^i\\
    &=\sum_{\emptyset\ne I\subset[m]}\frac{i!}{(i-|I|)!}\left(\sum_{k=1}^nx_k^{i-|I|}\times y_k^I\right)\times t_I\\
    &=\sum_{\emptyset\ne I\subset[m]}\frac{i!}{(i-|I|)!}\Phi^*(z(i-|I|+1,I))\times t_I.
  \end{align*}
  Thus the $G=U(n)$ case is proved. The $G=SU(n)$ case follows immediately from the $G=U(n)$ case, and the $G=Sp(n),SO(2n+1)$ case can be proved verbatim.
\end{proof}

Now we are ready to prove Theorem \ref{main 1}.

\begin{proof}
  [Proof of Theorem \ref{main 1}]
  Combine Theorem \ref{generator} and Proposition \ref{Theta(z)}.
\end{proof}


\subsection{Proof of Theorem \ref{nilpotent group}}

We show a property of the rational cohomology of a nilpotent group that we are going to use. We refer to \cite{HMR} for the localization of nilpotent groups. For a finitely generated group $\pi$, let $\overline{\mathrm{ab}}\colon\pi\to\Z^m$ denote the composite of the abelianization $\pi\to\pi^\mathrm{ab}$ and the projection $\pi^\mathrm{ab}\to\pi^\mathrm{ab}/\mathrm{Tor}\cong\Z^m$, where $\mathrm{Tor}$ is the torsion part of $\pi^\mathrm{ab}$.

\begin{lemma}
  \label{nilpotent cohomology}
  Let $\pi$ be a finitely generated nilpotent group. Then the rationalization $\pi_{(0)}$ is abelian if and only if the map
  \[
    \overline{\mathrm{ab}}^*\colon H^2(B\Z^m)\to H^2(B\pi)
  \]
  is injective.
\end{lemma}

\begin{proof}
  By definition, the rationalization of $B\pi$ is rationally homotopy equivalent to an iterated principal $S^1$-bundles. Then as in \cite{H}, the minimal model of $B\pi$ is given by $(\Lambda(x_1,\ldots,x_n),d)$ for $|x_i|=1$ such that
  \[
    dx_1=\cdots=dx_m=0,\quad dx_k=\sum_{i,j<k}\alpha_{i,j}x_ix_j\ne 0\quad(k>m).
  \]
  Moreover, the minimal model of $B\Z^m$ is given by $(\Lambda(x_1,\ldots,x_m),d=0)$ such that the map $\overline{\mathrm{ab}}\colon B\pi\to B\Z^m$ induces the inclusion $(\Lambda(x_1,\ldots,x_m),d=0)\to(\Lambda(x_1,\ldots,x_n),d)$. Observe that $\pi_{(0)}$ is abelian if and only if the map $\overline{\mathrm{ab}}\colon B\pi\to B\Z^m$ is a rational homotopy equivalence. Then $\pi_{(0)}$ is abelian if and only if $m=n$, which is equivalent to the map $\overline{\mathrm{ab}}^*\colon H^2(B\Z^m)\to H^2(B\pi)$ is injective.
\end{proof}

Now we are ready to prove Theorem \ref{nilpotent group}.

\begin{proof}
  [Proof of Theorem \ref{nilpotent group}]
  By the naturality of the map $\Theta$, there is a commutative diagram
  \[
    \xymatrix{
      \Hom(\Z^m,G)_0\ar[r]^(.45)\Theta\ar[d]_{\overline{\mathrm{ab}}^*}&\map_*(B\Z^m,BG)_0\ar[d]^{\overline{\mathrm{ab}}^*}\\
      \Hom(\pi,G)_0\ar[r]^(.45)\Theta&\map_*(B\pi,BG)_0.
    }
  \]
  Bergeron and Silberman \cite{BS} proved that the left map is a homotopy equivalence. Since the rationalization of $BG$ is a product of Eilenberg-MacLane spaces, there is a rational homotopy equivalence
  \begin{equation}
    \label{Thom}
    \map_*(X,BG)_0\simeq_{(0)}\prod_{n-i\ge 1}^\infty K(H^i(X)\otimes\pi_n(BG),n-i)
  \end{equation}
  for any connected CW complex $X$, which is natural with respect to $X$ and $G$. In particular, since $\pi_4(BG)\cong\Z$, there is a monomorphism $\iota\colon H_2(X)\to QH^2(\map_*(X,BG)_0)$ which is natural with respect to $X$, where $QA$ denotes the module of indecomposables of an augmented algebra $A$. Then there is a commutative diagram
  \[
    \xymatrix{
      H_2(B\pi)\ar[d]^{\overline{\mathrm{ab}}_*}\ar[r]^(.36)\iota&QH^2(\map_*(B\pi,BG))\ar[r]^{\Theta^*}\ar[d]^{(\overline{\mathrm{ab}}^*)^*}&QH^2(\Hom(B\pi,G)_0)\ar[d]^{(\overline{\mathrm{ab}}^*)^*}_\cong\\
      H_2(B\Z^m)\ar[r]^(.36)\iota&QH^2(\map_*(B\Z^m,BG))\ar[r]^{\Theta^*}&QH^2(\Hom(B\Z^m,G)_0).
    }
  \]
  By Theorem \ref{generator} and Propositions \ref{H(map)} and \ref{Theta(z)}, the composite of the bottom maps is an isomorphism. Thus by Lemma \ref{nilpotent cohomology}, the statement is proved.
\end{proof}

\begin{proof}
  [Proof of Corollary \ref{nilmanifold}]
  It is well known that a nilmanifold $M$ is homotopy equivalent to the classifying space of a finitely generated torsion free nilpotent group. Thus by Theorem \ref{nilpotent group}, the proof is finished.
\end{proof}


\subsection{Proof of Theorem \ref{main 2}}

Before we begin the proof of Theorem \ref{main 2}, we consider the case of $SO(2n)$ for $n=2,3$. We need the following lemma.

\begin{lemma}
  \label{covering}
  Let $G,H$ be compact connected Lie groups. If there is a covering $G\to H$, then there is a commutative diagram
  \[
    \xymatrix{
      \Hom(\Z^m,BG)_0\ar[r]^(.47)\Theta\ar[d]&\map_*(B\Z^m,BG)_0\ar[d]\\
      \Hom(\Z^m,BH)_0\ar[r]^(.47)\Theta&\map_*(B\Z^m,BH)_0,
    }
  \]
  where the vertical maps are isomorphisms in rational cohomology and rational homotopy groups.
\end{lemma}

\begin{proof}
  Let $K$ be the fiber of the covering $G\to H$. Then $K$ is a finite subgroup of $G$ contained in the center. In particular, the map $BG\to BH$ is a rational homotopy equivalence, implying that the right map is a rational homotopy equivalence. As is shown in \cite{Go}, the left map is a covering map with fiber $K^m$, so it is an isomorphism in rational homotopy groups because the fundamental groups of $\Hom(\Z^m,G)_0$ and $\Hom(\Z^m,H)_0$ are abelian as in \cite{GPS}. It is also proved in \cite{KT1} that the left map is an isomorphism in rational cohomology, completing the proof.
\end{proof}

\begin{corollary}
  \label{SO low rank}
  For $n=2,3$, the map
  \[
    \Theta\colon\Hom(\Z^m,SO(2n))_0\to\map_*(B\Z^m,BSO(2n))_0
  \]
  is surjective in rational cohomology.
\end{corollary}

\begin{proof}
  By Lemma \ref{covering}, it is sufficient to prove the statement for $Spin(2n)$, instead of $SO(2n)$. Then since $Spin(4)\cong SU(2)\times SU(2)$ and $Spin(6)\cong SU(4)$, the proof is finished by Theorem \ref{main 1}.
\end{proof}

Now we begin the proof of Theorem \ref{main 2}. For a monomial $z=x_1^{i_1}\cdots x_n^{i_n}y_1^{I_1}\cdots y_n^{I_n}$ in $H^*(BT\times T^m)$, let
\[
  d(z)=(i_1+|I_1|,\ldots,i_n+|I_n|)
\]
where $I_1,\ldots,I_n\subset[m]$. If all entries of $d(z)$ are even (resp. odd), then we call a monomial $z$ even (resp. odd).

\begin{lemma}
  \label{d}
  If $G=SO(2n)$, then every element of $H^*(BT\times T^m)^W$ is a linear combination of even and odd monomials.
\end{lemma}

\begin{proof}
  Given $1\le i<j\le n$, there is $w\in W$ such that
  \[
    w(x_k)=
    \begin{cases}
      -x_k&k=i,j\\
      x_k&k\ne i,j
    \end{cases}\qquad
    w(y_k)=
    \begin{cases}
      -y_k&k=i,j\\
      y_k&k\ne i,j.
    \end{cases}
  \]
  Then every monomial $z$ in $H^*(BT\times T^m)$ satisfies $w(z)=(-1)^{d_i+d_j}z$, where $d(z)=(d_1,\ldots,d_n)$. So if $z$ is contained in some element of $H^*(BT\times T^m)^W$, $d_1+d_2,d_2+d_3,\ldots,d_{n-1}+d_n$ are even. Thus $z$ is even for $d_1$ even, and $z$ is odd for $d_1$ odd, completing the proof.
\end{proof}

We define a map
\[
  \pi\colon H^*(BT\times T^m)\to H^*(BT\times T^m)^W,\quad x\mapsto\sum_{w\in W}w(x).
\]
For $m\ge 3$ and $G=SO(2n)$ with $n\ge 4$, let
\[
  \bar{a}=x_1\cdots x_{n-4}y_{n-3}^1y_{n-2}^2y_{n-1}^3y_n^1y_n^2y_n^3\in H^*(BT\times T^m).
\]
and let $a=\pi(\bar{a})$.

\begin{lemma}
  \label{a}
  The element $(\alpha\times 1)^*(a)$ of $H^*(SO(2n)/T\times T^m)^W$ is indecomposable, where $\alpha\colon G/T\rightarrow BT$ is as in Section \ref{the map}.
\end{lemma}

\begin{proof}
  It is easy to see that $\alpha^*(x_1\cdots x_{n-4})\ne 0$ in $H^*(SO(2n)/T)$ because
  \[
    H^*(SO(2n)/T)=\Q[x_1,\ldots,x_n]/(p_1,\ldots,p_{i-1},e)
  \]
  where $p_i$ is the $i$-th elementary symmetric polynomial in $x_1^2,\ldots,x_n^2$ and $e=x_1\cdots x_n$. Then $(\alpha\times 1)^*(\bar{a})\ne 0$ in $H^*(SO(2n)/T\times T^m)$. So since $a$ includes the term $2^{n-1}(n-4)!\bar{a}$, we have $(\alpha\times 1)^*(a)\ne 0$ in $H^*(SO(2n)/T\times T^m)^W$.

  Now we suppose that $(\alpha\times 1)^*(a)$ is decomposable. Then there are $b,c\in \widetilde{H}^*(BT\times T^m)$ such that $\pi(b)\pi(c)$ includes the monomial $\bar{a}$, and so we may assume $\bar{a}=bc$. Note that
  \[
    (1,\ldots,1,3)=d(\bar{a})=d(bc)=d(b)+d(c).
  \]
  Then since $d(b)\ne 0$ and $d(c)\ne 0$, it follows from Lemma \ref{d} that we may assume $d(b)=(1,\ldots,1)$, implying $b=x_1\cdots x_{n-4}y_{n-3}^1y_{n-2}^2y_{n-1}^3y_n^i$ for some $i=1,2,3$. Let $\sigma$ be the transposition of $n$ and $k$, where $k=n-3,n-2,n-1$ for $i=1,2,3$, respectively. Then $\sigma$ belongs to $W$, and $\sigma(b)=b$. Let $W=V\sqcup V\sigma$ be the coset decomposition. Then we have
  \[
    \pi(b)=\sum_{v\in V}v(b+\sigma(b))=\sum_{v\in V}v(b-b)=0
  \]
  and so we get $(\alpha\times 1)^*(a)=0$, which is a contradiction. Thus we obtain that $(\alpha\times 1)^*(a)$ is indecomposable, as stated.
\end{proof}

\begin{proposition}
  \label{SO(2n)}
  If $m\ge 3$ and $n\ge 4$, then $(\alpha\times 1)^*(a)\in H^*(SO(2n)/T\times T^m)^W$ does not belong to the image of the composite
  \[
    SO(2n)/T\times_W T^m\xrightarrow{\Phi}\Hom(\Z^m,SO(2n))_0\xrightarrow{\Theta}\map_*(B\Z^m,BSO(2n))_0
  \]
  in rational cohomology.
\end{proposition}

\begin{proof}
  First, we consider the $m=3$ case. Suppose that there is $\hat{a}\in H^*(\map_*(B\Z^3,BSO(2n))_0)$ such that $(\alpha\times 1)^*(a)=\Phi^*(\Theta^*(\hat{a}))$. Then by Lemma \ref{a}, $\Phi^*(\Theta^*(\hat{a}))$ is indecomposable. On the other hand, by Lemma \ref{Phi}, we have $\Phi^*(\Theta^*(\hat{a}))=\Theta^*(\widehat{\Phi}^*(\hat{a}))=\widehat{\Phi}(\hat{a})$, and by Proposition \ref{H(map)}, every indecomposable element of the image of $\widehat{\Phi}^*$ cannot contain a monomial $x_1^{i_1}\cdots x_n^{i_n}y_1^{I_1}\cdots y_n^{I_n}$ with $|I_1|+\cdots+|I_n|>4$. Thus we obtain a contradiction, so $(\alpha\times 1)^*(a)$ does not belong to the image of $\Phi^*\circ\Theta^*$.

  Next, we consider the case $m>3$. Since $\Z^3$ is a direct summand of $\Z^m$, the maps $\widehat{\Phi}$ and $\Theta$ for $m=3$ are homotopy retracts of the maps $\widehat{\Phi}$ and $\Theta$ for $m>3$, respectively. Thus the $m=3$ case above implies the $m>3$ case, completing the proof.
\end{proof}

Now we are ready to prove Theorem \ref{main 2}.

\begin{proof}
  [Proof of Theorem \ref{main 2}]
  Combine Theorem \ref{Baird} and Proposition \ref{SO(2n)}.
\end{proof}



\section{Rational homotopy groups}\label{rational homotopy groups}

This section proves Theorem \ref{main 3}. We begin with a simple lemma. Let $\hur^*\colon H^*(X)\to\Hom(\pi_*(X),\Q)$ denote the dual Hurewicz map. As in the proof of Theorem \ref{nilpotent group}, let $QA$ denote the module of indecomposables of an augmented algebra $A$. We refer to \cite{FHT} for rational homotopy theory.

\begin{lemma}
  \label{n+1}
  Let $X$ be a simply-connected space such that there is a map
  \[
    X\to\prod_{i=2}^nK(V_i,i)
  \]
  which is a rational equivalence in dimension $\le n$, where $V_i$ is a $\Q$-vector space of finite dimension. Then for $i\le n+2$, the map
  \[
    \hur^*\colon QH^i(X)\to\Hom(\pi_i(X),\Q)
  \]
  is injective.
\end{lemma}

\begin{proof}
  The minimal model of $X$ in dimension $\le n$ is given by
  \[
    (\Lambda(V_2\oplus\cdots\oplus V_n),d=0)
  \]
  where $\Lambda V$ denotes the free commutative graded algebra generated by a graded vector space $V$ and each $V_i$ is of degree $i$. Then there is no element of degree one in the minimal model of $X$, so any element of $QH^i(X)$ for $i\le n+2$ is represented by an indecomposable element of the minimal model of $X$. Since the module of indecomposables of the minimal model of $X$ is isomorphic to $\Hom(\pi_*(X),\Q)$ through the dual Hurewicz map, the proof is finished.
\end{proof}

We recall a property of the minimal generating set $\mathcal{S}(m,G)$ that we are going to use. Let
\[
  d(m,G)=
  \begin{cases}
    2n-m&G=U(n),SU(n)\\
    2n+1&G=Sp(n),SO(2n+1).
  \end{cases}
\]
Let $\Q\{S\}$ denote the graded $\Q$-vector space generated by a graded set $S$. We consider a map
\[
  \lambda=\prod_{x\in\mathcal{S}(m,G)}x\colon\Hom(\Z^m,G)_0\to \prod_{x\in\mathcal{S}(m,G)}K(\Q,|x|).
\]
The following is proved in \cite{KT1}.

\begin{theorem}
  \label{low dimension}
  Let $G$ be the classical group except for $SO(2n)$. Then the map
  \[
    \lambda^*\colon\Lambda(\Q\{\mathcal{S}(m,G)\})\to H^*(\Hom(\Z^m,G)_0)
  \]
  is an isomorphism in dimension $\le d(m,G)$.
\end{theorem}

We define a map $\hur^*\colon\mathcal{S}(m,G)\to\Hom(\pi_*(\Hom(\Z^m,G)_0),\Q)$ by the linear part of the map $\lambda$ in the minimal models.

\begin{lemma}
  \label{generator pi(G)}
  If $G$ is the classical group except for $SO(2n)$, then the map
    \[
      \hur^*\colon\Q\{\mathcal{S}(m,G)\}\to\Hom(\pi_*(\Hom(\Z^m,G)_0),\Q)
    \]
    is injective in dimension $\le d(m,G)+2$.
\end{lemma}

\begin{proof}
  By \cite{GPS}, $\Hom(\Z^m,G)_0$ is simply-connected whenever $G$ is simply-connected. Then by Lemma \ref{covering}, we may assume $\Hom(\Z^m,G)_0$ is simply-connected for $G=SU(n),Sp(n),SO(2n+1)$ as long as we consider rational cohomology and rational homotopy groups. By Theorem \ref{low dimension}, the map $\lambda$ is an isomorphism in rational cohomology in dimension $\le d(m,G)$. Then the statement for $G=SU(n),Sp(n),SO(2n+1)$ is proved by the J.H.C. Whitehead theorem and Lemma \ref{n+1}. For $G=U(n)$, we may consider $S^1\times SU(n)$ by Lemma \ref{covering}, instead of $U(n)$. In this case, the dual Hurewicz map for $G=U(n)$ is identified with the map
  \[
    1\times\hur^*\colon\Q^m\times\Q\{\mathcal{S}(m,SU(n))\}\to\Q^m\times\Hom(\pi_*(\Hom(\Z^m,SU(n))_0),\Q)
  \]
  because $\Hom(\Z^m,S^1\times SU(n))_0=(S^1)^m\times\Hom(\Z^m,SU(n))_0$. Thus the statement follows from the $G=SU(n)$ case.
\end{proof}

\begin{lemma}
  \label{generator pi(U)}
  For $G=U(n),SU(n)$, the map
    \[
      \hur^*\colon\Q\{\mathcal{S}(m,G)\}\to\Hom(\pi_*(\Hom(\Z^m,G)_0),\Q)
    \]
    is injective.
\end{lemma}

\begin{proof}
  Let $G=U(n),SU(n)$. We induct on $m$. If $m=1$, then the statement is obvious. Assume that the statement holds less than $m$. Take any $\emptyset\ne I\subset[m]$. Then there are the obvious inclusion $\iota_I\colon\Z^{|I|}\to\Z^m$ and the obvious projection $\pi_I\colon\Z^m\to\Z^{|I|}$ such that $\pi_I\circ\iota_I=1$. In particular, we get maps $\iota_I^*\colon\Hom(\Z^m,G)_0\to\Hom(\Z^{|I|},G)_0$ and $\pi_I^*\colon\Hom(\Z^{|I|},G)_0\to\Hom(\Z^m,G)_0$ such that $\iota_I^*\circ\pi_I^*=1$. Note that the map $\pi_I^*$ induces a map $(\pi_I^*)^*\colon\mathcal{S}(m,G)\to\mathcal{S}(|I|,G)$ such that
  \begin{equation}
    \label{top}
    (\pi_I^*)^*(z(d,J))=
    \begin{cases}
      z(d,J)&J\subset I\\
      0&J\not\subset I.
    \end{cases}
  \end{equation}
  Then there is a commutative diagram
  \begin{equation}
    \label{hur}
    \xymatrix{
      \Q\{\mathcal{S}(m,G)\}\ar[r]^{(\pi_I^*)^*}\ar[d]_{\hur^*}&\Q\{\mathcal{S}(|I|,G)\}\ar[d]^{\hur^*}\\
      \Hom(\pi_*(\Hom(\Z^m,G)_0),\Q)\ar[r]^{((\pi_I^*)_*)^*}&\Hom(\pi_*(\Hom(\Z^{|I|},G)_0),\Q).
    }
  \end{equation}
  Now we assume
  \[
    \sum_{z(d,J)\in\mathcal{S}(m,G)}a_{d,J}\hur^*(z(d,J))=0
  \]
  for $a_{d,J}\in\Q$. Then by \eqref{top} and \eqref{hur}, we have
  \begin{align*}
    0&=((\pi_I^*)_*)^*\left(\sum_{z(d,J)\in\mathcal{S}(m,G)}a_{d,J}\hur^*(z(d,J))\right)\\
    &=\sum_{z(d,J)\in\mathcal{S}(m,G)}a_{d,J}\hur^*(\pi_I^*)^*(z(d,J)))\\
    &=\sum_{\substack{z(d,J)\in\mathcal{S}(m,G)\\J\subset I}}a_{d,J}\hur^*(z(d,J))\\
    &=\sum_{z(d,J)\in\mathcal{S}(|I|,G)}a_{d,\iota_I(J)}\hur^*(z(d,J)).
  \end{align*}
  So since the right map of \eqref{hur} is injective for $I\ne[m]$ by the induction hypothesis, we get $a_{d,J}=0$ for $J\ne[m]$, implying
  \[
    \sum_{z(d,[m])\in\mathcal{S}(m,G)}a_{d,[m]}\hur^*(z(d,[m]))=0.
  \]
  Note that every $z(d,[m])\in\mathcal{S}(m,G)$ is of degree $\le 2n-m+1$. Then by Lemma \ref{generator pi(G)}, we get $a_{d,[m]}=0$, completing the proof.
\end{proof}

Now we prove Theorem \ref{main 3}.

\begin{proof}
  [Proof of Theorem \ref{main 3}]
  Let $\mathcal{S}_i$ and $\mathcal{S}_i(m,G)$ denote the degree $i$ parts of $\mathcal{S}$ and $\mathcal{S}(m,G)$, respectively, where $\mathcal{S}$ is as in Proposition \ref{H(map)}. Then by Proposition \ref{H(map)} and Theorem \ref{low dimension}, there is a commutative diagram
  \[
    \xymatrix{
      \Q\{\mathcal{S}_i\}\ar[r]^(.45){\Theta^*}\ar[d]_{\hur^*}&\Q\{\mathcal{S}_i(m,G)\}\ar[d]^{\hur^*}\\
      \Hom(\pi_i(\map_*(B\Z^m,BG)_0),\Q)\ar[r]^(.53){(\Theta_*)^*}&\Hom(\pi_i(\Hom(\Z^m,G)_0),\Q).
    }
  \]
  Let $K_i$ denote the kernel of the bottom map. Clearly, the dimension of $K_i$ coincides with
  \[
    \dim\mathrm{Coker}\{\Theta_*\colon\pi_i(\Hom(\Z^m,G)_0)\otimes\Q\to\pi_*(\map_i(B\Z^m,BG)_0)\otimes\Q\}
  \]
  and so we compute $\dim K_i$. By Proposition \ref{H(map)}, the left map is an isomorphism. Then we get
  \[
    \dim K_i\ge\dim \Q\{\mathcal{S}_i\}-\dim \Q\{\mathcal{S}_i(m,G)\}.
  \]
  By Lemma \ref{generator pi(G)}, the equality holds for $G=Sp(n),SO(2n+1)$ and $i\le d(m,G)+2$, and by Lemma \ref{generator pi(U)}, the equality holds for $G=U(n),SU(n)$ and all $i$. We can easily compute
  \[
    \dim \Q\{\mathcal{S}_i\}-\dim \Q\{\mathcal{S}_i(m,G)\}=
    \begin{cases}
      \displaystyle\sum_{i<k\le n}\binom{m}{2k-i}&G=U(n),SU(n)\\
      \displaystyle\sum_{\frac{i}{3}<k\le n}\binom{m}{4k-i}&G=Sp(n),SO(2n+1)
    \end{cases}
  \]
  and thus the proof is finished.
\end{proof}

\end{document}